\documentclass[11pt]{article}
\usepackage{cite}
\usepackage{mathrsfs} 
\usepackage{ifthen}
\usepackage{fancyhdr}
\usepackage[us,12hr]{datetime} 
\usepackage{exscale}
\usepackage{tabularx}
\usepackage{latexsym}
\usepackage{amsmath,amssymb,amstext,amsthm} 
\usepackage{matlab-prettifier}     
\usepackage{graphicx,color,epsfig}
\usepackage{graphicx}
\usepackage{fullpage}        
\usepackage{float}
\usepackage{verbatim}
\usepackage{enumerate}
\usepackage{makeidx}
\usepackage{tikz}
\usepackage[final]{pdfpages}
\usepackage{algorithm}
\usepackage{algpseudocode}
\usepackage{lineno}
\usepackage[bookmarks,pagebackref,
    pdfpagelabels=true, 
    ]{hyperref}
\numberwithin{equation}{section}
\usepackage[noabbrev]{cleveref}
\usepackage{multirow}
\numberwithin{table}{section}
\numberwithin{algorithm}{section}
\makeindex

\def\R{{\mathbb{R}}}
\def\Sc{\mathbb{S}}
\def\Sn{\Sc^n}

\def\Snp{\Sc_+^n}

\def\Rn{{\mathbb{R}^n}}

\newtheorem{theorem}{Theorem}[section]

\newtheorem{prop}[theorem]{Proposition}

\newtheorem{remark}[theorem]{Remark}

\newtheorem{problem}[theorem]{Problem}

\newtheorem{lemma}[theorem]{Lemma}

\crefname{thm}{Theorem}{Theorems}
\Crefname{thm}{Theorem}{Theorems}
\crefname{assump}{Assumption}{Theorems}
\Crefname{assump}{Assumption}{Theorems}
\crefname{problem}{Problem}{Theorems}
\Crefname{problem}{Problem}{Theorems}
\crefname{conjecture}{Conjecture}{Theorems}
\Crefname{conjecture}{Conjecture}{Theorems}
\crefname{proposition}{Proposition}{Propositions}
\Crefname{proposition}{Proposition}{Propositions}
\crefname{prop}{Proposition}{Propositions}
\Crefname{prop}{Proposition}{Propositions}
\crefname{cor}{Corollary}{Corollaries}
\Crefname{cor}{Corollary}{Corollaries}
\crefname{lem}{Lemma}{Lemmas}
\Crefname{lem}{Lemma}{Lemmas}
\theoremstyle{definition}
\crefname{definition}{definition}{definitions}
\Crefname{definition}{Definition}{Definitions}
\crefname{defn}{definition}{definitions}
\Crefname{defn}{Definition}{Definitions}
\crefname{remark}{Remark}{Remarks}
\Crefname{remark}{Remark}{Remarks}
\crefname{rmk}{Remark}{Remarks}
\Crefname{rmk}{Remark}{Remarks}
\crefname{example}{Example}{Examples}
\Crefname{example}{Example}{Examples}
\crefname{align}{}{}
\Crefname{align}{}{}
\crefname{equation}{}{}
\Crefname{equation}{}{}

\newcommand{\textdef}[1]{\textit{#1}\index{#1}}

\newcommand{\A}{{\mathcal A}}

\newcommand{\bbm}{\begin{bmatrix}}
\newcommand{\ebm}{\end{bmatrix}}
\newcommand{\bem}{\begin{pmatrix}}
\newcommand{\eem}{\end{pmatrix}}
\newcommand{\beq}{\begin{equation}}
\newcommand{\beqs}{\begin{equation*}}
\newcommand{\bet}{\begin{table}}
\newcommand{\eeq}{\end{equation}}
\newcommand{\eeqs}{\end{equation*}}
\newcommand{\beqr}{\begin{eqnarray}}

\DeclareMathOperator{\trace}{{trace}}

\DeclareMathOperator{\diag}{{diag}}
\DeclareMathOperator{\Diag}{{Diag}}

\DeclareMathOperator{\rank}{{rank}}
\DeclareMathOperator{\spanl}{{span}}

\newcommand{\nc}{\newcommand}
\nc{\arrow}{{\rm arrow\,}}
\nc{\Arrow}{{\rm Arrow\,}}
\nc{\BoDiag}{{\rm B^0Diag\,}}
\nc{\bodiag}{{\rm b^0diag\,}}

\nc{\Mm}{{\mathcal M}^{m} }
\nc{\Mmn}{{\mathcal M}^{mn} }
\nc{\Mnr}{{\mathcal M}_{nr} }
\nc{\Mnmr}{{\mathcal M}_{(n-1)r} }
\nc{\kwqqp}{Q{$^2$}P\,}
\nc{\kwqqps}{Q{$^2$}Ps}

\nc{\notinaho}{(X,S)\in \overline{AHO}(\A)}
\nc{\inaho}{(X,S)\in AHO(\A)}

\newcommand{\bea}{\begin{eqnarray}}%
\newcommand{\eea}{\end{eqnarray}}%
\newcommand{\beas}{\begin{eqnarray*}}%
\newcommand{\eeas}{\end{eqnarray*}}%
\newcommand{\Rnn}{\R^{n \times n}}%
{}

\newcommand{\Hnp}[1][]{\,\mathbb{H}_+^{\ifthenelse{\equal{#1}{}}{n}{#1}}}
\newcommand{\Hn}[1][]{\,\mathbb{H}^{\ifthenelse{\equal{#1}{}}{n}{#1}}}
\newcommand{\Dn}[1][]{\,\mathbb{D}^{\ifthenelse{\equal{#1}{}}{n}{#1}}}

\DeclareMathOperator{\st}{s.t.}

\begin{document}

\title{Finding Maximum Determinant Principal Submatrices via
Hadamard Bounds and Projection Methods}

\author{
Hao Hu\thanks{Corresponding author. School of Mathematical and Statistical Sciences, Clemson University, Clemson, SC 29634, USA. Email: \texttt{hhu2@clemson.edu}. Research supported by the Air Force Office of Scientific Research under award FA9550-23-1-0508.}
\and
Stefan Sremac\thanks{Department of Mathematics, Walla Walla University, 204 S College Ave, College Place, WA 99324, USA.}
\and
Hugo J. Woerdeman\thanks{Department of Mathematics, Drexel University, 3141 Chestnut Street, Philadelphia, PA 19104, USA. Research supported by the National Science Foundation under grant DMS 2348720.}
\and
Henry Wolkowicz\thanks{Department of Combinatorics and Optimization, Faculty of Mathematics, University of Waterloo, Waterloo, Ontario N2L 3G1, Canada. Research supported by the Natural Sciences and Engineering Research Council of Canada.}
}

\date{}

\maketitle

\begin{abstract}
An important yet challenging problem in numerical linear algebra is
finding a principal submatrix with maximum determinant from a given
symmetric positive semidefinite matrix. This problem arises in
experimental design, statistics, and machine learning.

We study several exact and approximate approaches to this
problem. We first derive an upper bound based on Hadamard's inequality,
along with a projection scheme based on the Gram--Schmidt process
without normalization. This combination yields a highly effective upper
bound and leads to an exact branch-and-bound algorithm for 
moderate-sized instances.

For larger scale problems
we propose a continuous relaxation that facilitates
reliable performance evaluation when the exact method returns only
near-optimal solutions. We further prove that the projection scheme
strengthens the upper bound derived from this relaxation. 
Numerical experiments demonstrate the effectiveness of the proposed methods across a broad range of datasets.
\end{abstract}


\section{Introduction}
The problem of identifying submatrices with maximum determinant arises
naturally in several areas of mathematics including combinatorics,
numerical linear algebra, experimental design, machine learning, and
statistics. Specifically, given a real symmetric positive semidefinite
matrix \( M \in \mathbb{R}^{n \times n} \), a principal submatrix of \(
M \) is obtained by selecting a subset of indices \( K \subseteq \{1,
\dots, n\} \) and extracting the corresponding rows and columns.
The task of selecting a subset \( K \) of fixed cardinality \(r \)
such that the resulting \textdef{principal submatrix, $M_K$}, has
maximum determinant is known as the \emph{maximum determinant principal
submatrix problem}, or \emph{MAXDET}. 
In this paper, we focus on the case where the size $r$ coincides with
the rank of the positive semidefinite matrix $M$. 

In numerical linear algebra,
finding well-conditioned submatrices in large-scale matrices is
of importance, and selecting a submatrix with large determinant modulus is
a useful relaxation. Note that this follows from the fact that
the ratio of the arithmetic and geometric means of the eigenvalues, 
$\omega(M) = \frac {\trace(M)/n}{\det(M)^{1/n}}$, 
is a valid condition number, see e.g.,~\cite{jung2025omega}.
\index{$A_S$, principal submatrix}

This problem is computationally challenging: it is NP-hard in
general~\cite{ko1983computational, di2015largest}, and even
approximating the optimal value is significantly
challenging~\cite{khachiyan1995complexity}. The problem is also
closely related to volume maximization in convex geometry, such
as finding the maximum volume of a simplex or parallelepiped spanned 
by a subset of vectors~\cite{goreinov1997theory, goreinov2001maximal,
di2015largest}. These geometric interpretations link MAXDET to low-rank
approximation~\cite{goreinov2010good} and matrix sketching techniques.

In experimental design, particularly for linear and logistic regression,
selecting a submatrix with maximum determinant corresponds to finding a
D-optimal design, one that minimizes the volume of the confidence
ellipsoid for the parameter estimates in
regression models~\cite{pukelsheim2006optimal}. Applications include
medical data modeling~\cite{ouyang2016design}, constrained design
construction~\cite{ucinski2015algorithm}, and categorical data
analysis~\cite{yang2017d}. Classical and Bayesian formulations of
optimal design problems have been well studied in the
literature~\cite{sebastiani2000maximum}.
In particular, when the data matrix $V \in \R^{n \times r}$ has full
column rank and one selects $s = r$ rows so that
$M = VV^T$ and $\det(M_K) = \det(V_K)^2$, our MAXDET problem
coincides with the \emph{0/1 D-optimality problem}~\cite{welch1982branch,
ko1998comparison}. This is also closely related to the
\emph{maximum-entropy sampling problem (MESP)}, in which one
selects an order-$s$ principal submatrix of a positive
definite covariance matrix $C$ so as to maximize $\log\det C[S,S]$;
see~\cite{ko1995exact, lee1998constrained, anstreicher1999using,
anstreicher2018maximum, anstreicher2020efficient}, the 
monograph~\cite{fampa2022maximum}, and the recent
survey~\cite{fampa2026recent}. The two problems are unified by the
generalized maximum-entropy sampling framework
of~\cite{ponte2026convex}, and recent work
\cite{ponte2025relationship,ponte2024admm} establishes precise
mappings between MESP and 0/1 D-Opt and develops fast
ADMM solvers for the corresponding convex relaxations.
Our setting in this paper, where $M$ is positive
\emph{semi}definite of rank exactly $r$ and the submatrix size equals
$r$, falls in the intersection of these two lines of work, but is
also distinct as it is driven directly by the rank-decomposition factor
$V$ rather than the (potentially much larger) $n \times n$ matrix $M$.

In machine learning and statistics, determinant-based selection criteria have been adopted for data summarization and diversity modeling using determinantal point processes~\cite{kulesza2012determinantal}. Related problems also arise in active learning~\cite{zheng2023batch} and sparse modeling with budget constraints~\cite{ravi2016experimental}.

To cope with the combinatorial explosion of subset selection, various
algorithmic approaches have been proposed. These include
greedy algorithms based on rank-revealing
QR factorizations, convex relaxations using log-determinant
functions~\cite{candes2009exact, della1982minimum}, and matrix
inequality constrained optimization frameworks, such as those in
determinant maximization problems~\cite{vandenberghe1998determinant}.
Recent advances include primal-dual
methods for optimization~\cite{gonzalez2009stable}, submodular function
maximization under matroid or partition
constraints~\cite{nikolov2015proportional, nikolov2016maximizing}, and
proportional volume sampling~\cite{nikolov2015proportional}. More recent
advances have focused on developing efficient upper bounds, 
such as the `linx' bound, which can be used within these frameworks 
to accelerate pruning and improve solution efficiency 
\cite{anstreicher2020efficient,Chenetal2024}.

Our work concentrates on four points:
\textup{(i)}~a Hadamard-type upper bound for the maximum determinant
principal submatrix problem;
\textup{(ii)}~orthogonal projections that materially tighten both that bound and
a continuous determinant relaxation;
\textup{(iii)}~integrating these components in branch-and-bound, which yields a highly
effective exact procedure for small- to moderate-sized instances and, on larger
problems, still pairs favorably with~\textup{(i)} for pruning; and
\textup{(iv)}~extensive numerical experiments, including formulations derived from NP-hard odd cycle packing, validate the approach.

\textbf{Notation:} We use $|\cdot|$ to denote both absolute value and 
cardinality, depending on the context.
We define \textdef{$\diag(M)$}$: \Rnn \to \Rn$ to be the linear transformation on
square matrices of order $n$ that yields the diagonal vector;
the adjoint linear transformation is \textdef{$\Diag(v)=\diag^*(v)$}. We
denote by $e$ the vector of all ones of appropriate dimension. 
Further notation is introduced below as needed.
\index{ones vector, $e$}
\index{$e$, ones vector}

\textbf{Outline:} The paper is organized as follows.
We continue in~\Cref{sect:prel} with the problem definition.
Then in~\Cref{sect:reforproj} we reformulate the problem using projections, which substantially strengthen the upper bounds. 
 In~\Cref{sect:upperbnds} 
we present an efficient upper bound for the MAXDET problem
based on the Hadamard inequality. This upper bound together with the projection technique yields an efficient branch-and-bound algorithm for solving the MAXDET problem. We also study the upper bound based on a continuous relaxation, and its behavior under the projection technique. In \Cref{sect:numerics}, we present the numerical results of our branch-and-bound algorithm and the continuous relaxation upper bounds.

\section{Preliminaries}
\label{sect:prel}

\subsection{Problem definition}
\index{$\Sn$, symmetric matrices order $n$}
\index{symmetric matrices order $n$, $\Sn$}
Let $\Sn$ be the Euclidean space of real $n\times n$ symmetric matrices
equipped with the \textdef{trace inner-product} 
$\langle A,B\rangle = \trace(AB)$,
and let $\Snp$ denote the cone of positive semidefinite matrices.  We fix
the given data:
\begin{equation}
	\label{eq:L}
	M \in \Snp, \text{  with  }  \rank(M) = r<n. 
\end{equation}
\index{$M_K$,  principal submatrix}
\index{ principal submatrix, $M_K$}
For any subset $K \subseteq \{1,\dotso,n\}$,
we let $M_K$  denote the principal submatrix of $M$ consisting of 
elements from rows and columns indexed by $K$. We consider the following 
problem of maximizing the volume (or determinant).
\index{$\Sn$, symmetric matrices} 
\index{symmetric matrices, $\Sn$} 
\index{$\Snp$, positive semidefinite matrix cone} 
\index{positive semidefinite matrix cone, $\Snp$} 

\begin{problem}[{\textdef{MAXDET}}]
\label{prob:main}
Given $M,r,n$ as in \cref{eq:L}, solve:
\begin{equation}
\label{eq:main}
\max \, \{ \det(M_K) \,:\, K \subseteq \{1,\dotso,n\}, \,\lvert K \rvert = r \}.
\end{equation}	
\end{problem}
We note that: in \Cref{prob:main} the matrix $M$ is positive
\emph{semi}definite, without loss of generality $\rank(M)=r < n$, 
and the submatrix size
equals the rank. This is in contrast to the setting
in~\cite{anstreicher2020efficient}, where $M$ is assumed to be positive
definite. 

It is convenient for our purposes to use an alternative formulation of
\Cref{prob:main} using a factorization of $M$.  Let $V \in
\R^{n\times r}$ provide the full rank factorization \textdef{$M = VV^T$}.   
Let \( V_K \) denote the submatrix of \( V \) consisting of the rows
indexed by \( K \).  Then we have $\det(M_{K}) = \det(V_{K})^{2}$ and
thus \Cref{prob:main} is equivalent to:
\begin{equation}
	\label{eq:main2}
\max \, \{ |\det(V_K)| \,:\, K \subseteq \{1,\dotso,n\}, \,\lvert K \rvert = r \}.
\end{equation}
In this paper we assume that the matrix $V$ in the factorization is given.
In many applications, such as experimental design, the data matrix $V$
is available directly and the symmetric matrix $M = VV^T$ is never
formed explicitly.
We also note that the optimal value of \eqref{eq:main} is the square of the optimal value of \eqref{eq:main2}.

To compute an exact optimal solution, a branch-and-bound algorithm can be employed to recursively determine the subset of rows of \( V \) to be included in the final solution.  
We consider a generalized version of the original problems \eqref{eq:main} and \eqref{eq:main2}, where a subset of rows is required to be included in the solution and certain rows may be excluded. Rows that are excluded can be safely removed from~\( V \). Given a subset \( J \subseteq \{1,\ldots,n\} \) of required rows with at most $r$ elements, define the following variants of \eqref{eq:main} and \eqref{eq:main2}.
\begin{enumerate}
	\item Find the principal submatrix of $M=VV^T$ that
includes the rows
and columns indexed by $J$ and that has maximum determinant value:
	\begin{equation}
		\label{eq:maina}
		\begin{array}{rrll}
 \textdef{$\delta(V, J)$} :=  
			&\max & \det(M_{J \cup K}) \\
			&\st & K \subseteq \{1,\ldots,n\}, \\
			&& J \cap K = \emptyset,\\
			&& |J \cup K| = r.
		\end{array}
	\end{equation}
	
	\item Find the 
\( r \times r \) submatrix of \( V \) that contains the rows indexed by \( J \)
and that has maximum absolute determinant value:
	\begin{equation}
		\label{def_rprob}
		\begin{array}{rrll}
	\sqrt{\delta(V, J)} =  
    &\max & |\det(V_{J \cup K})| \\
			&\st & K \subseteq \{1,\ldots,n\}, \\
			&& J \cap K = \emptyset,\\
			&& |J \cup K| = r.
		\end{array}
	\end{equation}
\end{enumerate}
Note that $ |J| $ can range from $ 0 $ to $ r $, and we require $ r - |J| $ additional rows to form a full-rank $ r \times r $ submatrix.

\section{Reformulation via Projection}
\label{sect:reforproj}
In this section, we describe a projection-based reformulation of the matrix \( V \) that is particularly useful within a branch-and-bound algorithm. This approach can significantly strengthen upper bounds in practice.

Without loss of generality, for \cref{eq:maina,def_rprob},
we assume \( J = \{1,\ldots,j\} \) for some \( j \leq r \). 
We define a
new matrix \( \widetilde{V} \in \mathbb{R}^{n \times r} \), whose rows
\( \tilde{v}_1, \ldots, \tilde{v}_n \) are computed via the following
orthogonal projection procedure in \Cref{alg:orthogproj}:

\begin{algorithm}[H]
  \small
	\caption{Orthogonal Projection Process
\\for $V\in \R^{n\times r},J=\{1,\ldots, j\}$}
	\label{alg:orthogproj}
	\begin{algorithmic}[1]
		\State Set \( \tilde{v}_1 \gets v_1 \).
		\For{$i = 2$ to $j$}
		\State Set \( \tilde{v}_i \) to the projection of \( v_i
\) onto \( \spanl \{\tilde{v}_1, \ldots, \tilde{v}_{i-1}\}^\perp \).
		\EndFor
		\For{$i = j+1$ to $n$}
		\State Set \( \tilde{v}_i \) to the projection of \( v_i
\) onto \( \spanl \{\tilde{v}_1, \ldots, \tilde{v}_j\}^\perp \).
		\EndFor
	\end{algorithmic}
\end{algorithm}

Since the determinant is preserved under this projection process, we have
$
|\det(V_{J \cup K})| = |\det(\widetilde{V}_{J \cup K})|
$
for all subsets \( K \subseteq \{1,\ldots,n\} \) such that \( J \cap K =
\emptyset \) and \( |J \cup K| = r \). It follows that \( \delta(V, J) =
\delta(\widetilde{V}, J) \). Although this is an equivalent
reformulation in terms of determinant values, we show below that various upper bounds on \( \delta(\widetilde{V}, J) \) are significantly tighter than the corresponding bounds on \( \delta(V, J) \).

\begin{remark}
The projections in Algorithm~\ref{alg:orthogproj} are computationally
inexpensive and admit further optimization in a branch-and-bound
setting.

\begin{enumerate}
	\item In line~3, many projections can be reused to avoid
redundant computation. For instance, suppose the current node in the
branch-and-bound tree corresponds to \( J = \{1,\ldots,j\} \), and the
matrix \( \widetilde{V} \) has already been computed. Since the vectors
\( \tilde{v}_1, \ldots, \tilde{v}_{j+1} \) are pairwise orthogonal, if
the algorithm branches by setting \( x_{j+1} = 1 \), thereby updating \(
J \leftarrow \{1,\ldots,j+1\} \) in the child node, then the projections
of \( v_i \) onto \( \spanl \{\tilde{v}_1, \ldots, \tilde{v}_{i-1}\}^\perp \) for \( i = 2,\ldots,j+1 \) need not be recomputed, provided they are cached appropriately from the parent node.
	
	\item In line~6, since \( \tilde{v}_1, \ldots, \tilde{v}_j \) are orthogonal, the projection of \( v_i \) onto the orthogonal complement of their span has the explicit formula:
	$
	\tilde{v}_i = v_i - \sum_{k=1}^{j} \frac{\tilde{v}_k^\top v_i}{\tilde{v}_k^\top \tilde{v}_k} \tilde{v}_k.
	$
	This expression is computationally efficient.
\end{enumerate}
\end{remark}

\begin{remark}[Role of the rank assumption]
\label{rem:projrank}
The projection reformulation is one of the main ingredients of our
approach: as we show in \Cref{sec:hada} and \Cref{sect:sdprelax},
without it, the Hadamard upper bound is typically very weak,
while the projected matrix $\widetilde{V}$ yields a substantially
tighter bound and similarly tightens the continuous relaxation.

It is important to emphasize that this technique relies in an essential
way on the assumption that the selected submatrix size equals the rank,
i.e., $|J \cup K| = r = \rank(V)$. Once $\widetilde{V}$ is constructed
from $V$ with respect to $J = \{1,\ldots,j\}$, the identity
$|\det(V_{J\cup K})| = |\det(\widetilde{V}_{J\cup K})|$ holds for every
admissible index set with $|J\cup K| = r$, but \emph{fails} for index
sets of cardinality strictly larger than $r$. Consequently, in the more
general 0/1 D-optimality and maximum-entropy sampling settings, where
one may select more than $r$ rows of $V$ (or equivalently, an order-$s$
principal submatrix of an $n\times n$ matrix with $s$ unrelated to the
rank), the projection step is not directly applicable. Our setting
$s = r$ is precisely what allows us to exploit this special structure.
\end{remark}

\section{Upper bounds}
\label{sect:upperbnds}
We now consider several efficient techniques for obtaining upper bounds. 
This includes using the Hadamard inequality as well as a continuous
relaxation.

\subsection{Upper bounds based on Hadamard inequality}
\label{sec:hada}
In this section, we present an efficient upper bound for the maximum
determinant problem. Our approach leverages the Hadamard
inequality~\cite{gjmw2}, which states that the absolute value of the
determinant of a square matrix is bounded above by the product of the
norms of its rows.\footnote{Note that the inequality for $M$ positive
definite is $\det(M) \leq \Pi_i M_{ii}$.} To tighten this bound, we also reformulate the rows of the matrix~$V$ via the projection technique as discussed in the last section.  Numerical results demonstrate that the proposed Hadamard-based upper bound is highly effective, enabling the identification of optimal solutions within reasonable time when embedded in a branch-and-bound framework.

Let \( V \in \mathbb{R}^{n \times r} \) and \( J \subseteq \{1,\ldots,n\} \) be given. For any subset \( K \subseteq \{1,\ldots,n\} \) such that
\begin{equation}
	\label{setK}
	J \cap K = \emptyset, \quad |J \cup K| = r,
\end{equation}
the submatrix \( V_{J \cup K} \) is of size \( r \times r \). Applying the Hadamard inequality yields the bound
\[
|\det(V_{J \cup K})| \leq \prod_{i \in J \cup K} \|v_i\|,
\]
where \( v_i \) denotes the \( i \)-th row of \( V \). This leads to the
following upper bound:
\[
\begin{array}{rcl}
 \delta(V, J)
&\leq &
	\delta_{H}(V, J),
\end{array}
\]
where
\[
\delta_{H}(V, J)=
\max \left\{ \prod_{i \in K} \|v_i\| \;\middle|\; K \text{ satisfies } \eqref{setK} \right\}
\prod_{i \in J} \|v_i\|.
\]

\begin{prop}
	Let \( V \in \mathbb{R}^{n \times r} \) be given, and let \( J \subseteq \{1,\ldots,n\} \) with \( |J| \leq r \). Let \( \widetilde{V} \in \mathbb{R}^{n \times r} \) be the matrix obtained by applying the orthogonal projection process described in Algorithm~\ref{alg:orthogproj} to \( V \) with respect to \( J \). Then,
	$
	\delta_{H}(\widetilde{V}, J) \leq \delta_{H}(V, J).
	$
\end{prop}
\begin{proof}
It follows from the description of Algorithm~\ref{alg:orthogproj} that \( \|\tilde{v}_i\| \leq \|v_i\| \) for all \( i \), the inequality follows.
\end{proof}

The bound $\delta_{H}(\widetilde{V},J)$ is typically much tighter than $\delta_{H}(V,J)$. In practice, the projection significantly improves the quality of the bound.

\subsection{A continuous relaxation}
\label{sect:sdprelax}

\index{LMI, linear matrix inequality}

For huge problem instances, we may not be able to find the optimal
solution. In this case, it is important to measure the current
(relative) gap \cref{eq:relGAP} between 
the best upper and lower bounds of the optimal value, where the best
lower bound is obtained from the current best feasible solution. This gives the users a good idea about the quality of the feasible solution. To this end, we derive and analyze stronger upper bounds based on convex optimization techniques.

Let $V \in \R^{n\times r}$ and $J \subseteq \{1,\ldots,n\}$ be given.
Define the set:
\begin{equation}
	\label{def_tri}
\triangle_{n} := \left\{  x \in \Rn \mid e^{T}x = r, \, 0\leq x \leq 1 \right\}.
\end{equation}
A binary programming formulation for the original problem \eqref{eq:maina} is given as follows:
\begin{equation}
	\label{eq:ilpdef}
	\begin{array}{rll}
		\max & \det(V^T \Diag(x) V)\\
		\st & x \in \triangle_{n} \cap \{0,1\}^n,\\
		& x_{i} = 1 \text{ for } i \in J.\\
	\end{array}
\end{equation}
Indeed, if $x^{\star}$ is an optimal
solution to \cref{eq:ilpdef}, then $K:= \{ i \in \{1,\dotso,n\} : 
x^{\star}_i = 1\}$ is optimal for \cref{eq:maina}.

If we discard the binary constraints, then we obtain the continuous
relaxation: 
\begin{equation}
	\label{eq:lp}
	\begin{array}{rcll}
\delta_{\rm cont}(V,J) := & \max & \det(V^T \Diag(x) V)\\
&	\st & x \in \triangle_{n},\\
&& x_{i} = 1 \text{ for } i \in J.\\
	\end{array}
\end{equation}
It follows immediately that
$\delta_{\rm cont}(V,J)$ provides an upper bound for the maximum of
\eqref{eq:ilpdef}. Indeed, any binary feasible solution is also feasible
for the continuous relaxation. We note that the function $\log\det(V^T
\Diag(x) V)$ is concave on its domain, and thus the maximization problem \eqref{eq:lp} is a convex optimization problem. Though the constraints are linear, the nonlinear objective means this is not a linear program.

\begin{remark}
The relaxation \eqref{eq:lp} is sometimes referred to
as the \emph{natural bound} for the 0/1 D-optimality problem; see,
e.g., \cite{welch1982branch,ko1998comparison,ponte2025relationship}.
Several other bounds, based on convex relaxations have been developed in recent years.  Most notably, the linx bound \cite{anstreicher2020efficient}, the BQP bound \cite{anstreicher2018maximum}, and the $\Gamma$-bound \cite{ponte2025branch}, \cite{nikolov2015randomized}. Fast ADMM algorithms for computing these bounds have been recently developed in~\cite{ponte2024admm}. Theoretical results showing that one of these bounds dominates another are lacking or limited to special cases \cite{ponte2025branch}.  Our choice of the natural bound is largely motivated by two factors. First, the natural bound is attractive computationally since the determinant takes matrices of order $r$ as opposed to $n$ or $n-r$ in the aforementioned bounds.  In our numerical results, many of the instances have $r \ll n$.  As observed in \cite{ponte2025branch}, the natural bound tends to perform better than the $\Gamma$-bound when $r < n/2$.  Secondly, the natural bound is invariant to scaling of the input matrix $M$.  In contrast, bounds such as linx and BQP are tightest when an optimal scaling is chosen for $M$.  This is a non-trivial step, as described in \cite{anstreicher2020efficient}, for instance.
\end{remark}

We prove that the projection technique also works for the continuous relaxation
\eqref{eq:lp}, namely, if $\widetilde{V}$ is the matrix obtained from
$V$ by applying the projection procedure described in
Algorithm~\eqref{alg:orthogproj}, then $\delta_{\rm cont}(\widetilde{V},J)
\leq \delta_{\rm cont}(V,J)$. To prove this, we derive the following useful
linear algebra results.

\begin{lemma}
	\label{L1}
	$\det \begin{pmatrix} \beta + a & b^T \cr c & D \end{pmatrix} =
\beta \det D +  \det \begin{pmatrix} a & b^T \cr c & D \end{pmatrix}. $
\end{lemma}
\begin{proof}
Note that	the first column can be written as $\beta \begin{pmatrix}
		1\\
		0
	\end{pmatrix} + \begin{pmatrix}
	a\\
	c
	\end{pmatrix}$.
	Applying the multilinearity of the determinant, we obtain the desired equation.
\end{proof}

\begin{prop}
	\label{P1} 
	Let $A \in {\mathbb R}^{k\times k}$ and $\begin{pmatrix} B & C \cr C^T & D \end{pmatrix}\in {\mathbb R}^{n\times n}$
	be positive semidefinite. 
	Then 
	$\det \begin{pmatrix} A + B & C \cr C^T & D \end{pmatrix} \ge \det A \det D.$
\end{prop}
\begin{proof} We proceed by induction on $k$, the size of $A$. When $k=1$, the result follows directly from Lemma \ref{L1}. Suppose that the result has been proven for $A$ up to size $k-1$.
	
Without loss of generality, we assume that $A$ is a diagonal matrix, i.e., $A={\rm Diag} (a_i)_{i=1}^k$, where $a_i \ge 0$, $i=1,\ldots , k$. Let $P:{\mathbb R}^{k} \to {\mathbb R}^{k-1}$ be the projection on the last $k-1$ entries. By the case $k=1$, we have that
	$$  \det \begin{pmatrix} A + B & C \cr C^T & D \end{pmatrix} \ge a_1 \det  \begin{pmatrix} PAP^T + PBP^T & CP^T \cr PC^T & D \end{pmatrix}. $$
	By the induction assumption, we get that
	$$ \det  \begin{pmatrix} PAP^T + PBP^T & CP^T \cr PC^T & D \end{pmatrix} \ge \det(PAP^T) \det(D) . $$
	Since, $\det A = a_1  \det(PAP^T)$, the result follows.
\end{proof}

Now we are ready to prove the main theoretical result.
\begin{theorem}
	Let $V \in \R^{n\times r}$ and $J \subseteq \{1,\ldots,n\}$ be given.	Let $\widetilde{V}$ be the matrix obtained from $V$ by applying the
projection procedure in Algorithm~\ref{alg:orthogproj}. Then 
$\delta_{\rm cont}(\widetilde{V} , J) \leq \delta_{\rm cont}(V,J).$
\end{theorem}
\begin{proof}
Without loss of generality, we assume that the rows in $V$ indexed by $J$ are linearly independent and $J =\{ 1, \ldots , k \}$ for some $k< r$. Note that the first $k$ rows of $\widetilde{V}$ are orthogonal. We may multiply $\widetilde{V}$ on the right with a unitary matrix $U$, and make the first $k$ rows of $\widetilde{V}$ to be of the form
	$\begin{pmatrix} \widetilde{\Delta} & 0 \end{pmatrix}$, where $\widetilde{\Delta}$ is a $k\times k$ positive definite diagonal matrix. 	Indeed, if $w_1^T, \ldots , w_k^T$ are the first $k$ rows of $V$ and are all nonzero, one would take 
	$$ U= \left( \frac{w_1}{\| w_1 \|} \;\cdots\;  \frac{w_k}{\| w_k \|} \; u_{k+1} \;\cdots\; u_r \right)\!, $$
	where $\{ u_{k+1} , \ldots , u_r \}$ is an orthonormal basis for 
$(\spanl \{ w_i \}_{i=1}^k)^\perp$.

	Thus we have that we may assume that
	$\widetilde{V} =\begin{pmatrix} \widetilde{\Delta} & 0 \cr 0 & B \end{pmatrix} , $
	for some $(n-k)\times (r-k)$ matrix $B$. The way, $  \widetilde{V}$ was constructed from $V$, gives that
	${V} =\begin{pmatrix} L\widetilde{\Delta} & 0 \cr A & B \end{pmatrix} , $
	where $L$ is a lower triangular matrix with $1$'s on the diagonal, and $A$ is of size $(n-k)\times k$.
	
	Suppose that $x^{*}$ yields the maximum $ \delta_{\rm cont}(\widetilde{V} , J)$. Note that $x_{j}^{*}=1$ for $j=1,\ldots ,k$. Let $y = (x_{i})_{i=k+1}^{n}$ be the vector consisting of the remaining $n-k$ entries of $x^{*}$. Thus $\delta_{\rm cont}(\widetilde{V} , J ) = \det \widetilde{\Delta}^2 \det (B^T {\rm Diag} (y) B).$
	Now 
	{\small
	$$\begin{array}{rll}
		\delta_{\rm cont}(V, J ) &\ge& \det (V^T {\rm Diag}(x^{*}) V ) \\[3pt]
		&=& \det \!\begin{pmatrix} \widetilde{\Delta}L^TL\widetilde{\Delta} \!+\! A^T {\rm Diag} (y) A & A^T {\rm Diag} (y) B \cr B^T {\rm Diag} (y) A & B^T {\rm Diag} (y) B \end{pmatrix} \\[3pt]
		&\ge& \det ( \widetilde{\Delta}L^TL\widetilde{\Delta}) \det (B^T {\rm Diag} (y) B)  \\[3pt]
		&=& \delta_{\rm cont}(\widetilde{V} , J) \\
	\end{array}$$
	}
	where the last inequality follows from Proposition~\ref{P1} and we used that $\det(\widetilde{\Delta}L^TL\widetilde{\Delta})= \det \widetilde{\Delta}^2$ due to $\det L=1$.

\end{proof}

\section{Numerics}
\label{sect:numerics}

In this section, we present our numerical experiments. All computations were performed on a Mac Studio (2023) equipped with an Apple M2 Ultra chip, 128~GB of RAM, and running macOS 14.1.1 (build 23B81). The branch-and-bound algorithm was implemented in MATLAB R2024b. The continuous relaxation \eqref{eq:lp} was formulated as a linear SDP and solved using MOSEK version 11~\cite{aps2019mosek}.

We report the objective value as \( \log\det(V^{\top} \Diag(x) V) \) or  \( \log_{2}\det(V^{\top} \Diag(x) V) \), as the determinant can take extremely large values in high-dimensional settings.

We solve the problem using the branch-and-bound algorithm, which
incorporates Hadamard-type inequality bounds described in
Section~\ref{sec:hada}. The best objective value obtained from the
branch-and-bound process is recorded as the \textbf{lower bound (LB)}.
Additionally, we solve the continuous relaxation \eqref{eq:lp} to obtain
an \textbf{upper bound (UB)} on the optimal value. Both bounds are
nonnegative.

To evaluate the quality of the solutions, we compute the normalized duality gap:
\begin{equation}
\label{eq:relGAP}
\text{GAP} = \frac{\text{UB} - \text{LB}}
            {\frac {|\text{UB}|+ |\text{LB}|}2  + 1}.
\end{equation}

A time limit of 10 minutes is imposed for the branch-and-bound
algorithm. The algorithm \emph{certifies optimality} when all nodes in
the branch-and-bound tree have been either explored or pruned, so that
no unexplored node can yield an objective value exceeding the current
best feasible solution. When this occurs within the time limit, the
corresponding LB is marked with an asterisk~(*) in the tables. For both
the branch-and-bound and continuous relaxation approaches, we report the
runtime in cpu seconds.

\subsection{The UCI Machine Learning Repository}

We evaluate our algorithm using datasets from the UCI Machine Learning
Repository \cite{frank2010uci}. Each dataset is represented by a matrix
\( V \in \mathbb{R}^{n \times r} \), where each row corresponds to an
observation and each column to a feature. We restrict our experiments to
datasets that satisfy the following two conditions: (1) the matrix \( V
\) contains only numerical values, and (2) the number of observations
exceeds the number of features, i.e., \( n > r \). If \( V \) contains
linearly dependent columns, we extract a subset of linearly independent
columns using the QR factorization. With slight abuse of notation, we continue to denote the resulting matrix by \( V \).

\Cref{tab:full_uci1,tab:full_uci2}, \cpageref{tab:full_uci1,tab:full_uci2}, 
present a representative subset of results. We report: problem
dimensions (\(n\), \(r\)), objective values (in log base $e$) obtained
by both methods (LB from branch-and-bound, UB from continuous
relaxation), computation times in seconds, and the relative duality gap
from \cref{eq:relGAP}.

\paragraph{Observations.}
Based on the numerical results, we highlight the following:
\begin{itemize}
	\item The branch-and-bound algorithm successfully solves several small- and medium-sized instances to optimality within seconds (e.g., \texttt{Lenses}, \texttt{Iris}, \texttt{Ecoli}, \texttt{Yeast}), as indicated by the asterisk next to the LB.
	\item For larger or more complex datasets (e.g., \texttt{Parkinsons}, \texttt{Image\_Segmentation}), the branch-and-bound method often reaches the time limit without proving optimality, although it still provides meaningful lower bounds.
	\item The continuous relaxation is highly efficient in most cases, typically completing in under one second. However, its quality varies significantly across datasets. In two large-scale instances, the solver exceeded our machine’s memory or time constraints; these are labeled as N.A.
	\item The duality gap is negligible when the branch-and-bound algorithm reaches optimality or a tight bound, indicating the relaxation is informative in those cases. Conversely, large duality gaps (e.g., \texttt{Vertebral\_Column}, \texttt{User\_Knowledge}) suggest weaker continuous relaxations and emphasize the need for exact methods.
\end{itemize}

Overall, the results demonstrate the effectiveness of the proposed approach on small to moderate-sized datasets and illustrate the trade-off between computational time and solution quality for larger problems.

\begin{table*}[!t]
	\centering
	\footnotesize
	\setlength{\tabcolsep}{2.85pt}%
	\renewcommand{\arraystretch}{0.86}%
	\begin{tabular}{|l|r|r|c|r|c|c|c|}
		\hline
		\multirow{2}{*}{Dataset} & \multirow{2}{*}{$n$} & \multirow{2}{*}{$r$} & \multicolumn{2}{c|}{Branch-and-Bound}   & \multicolumn{2}{c|}{Continuous Relaxation}   & \multirow{2}{*}{GAP} \\  \cline{4-7}
		& & & LB &  Time & UB & Time &  \\ 
		\hline
Challenger$\_$USA$\_$Space$\_$Shuttle$\_$O-Ring                             & $23$      & $4$     & $-17.3036(*)$ & $0.72$      & $-17.1441$      & $0.01$          & $0.01$         \\
Lenses                                                                      & $24$      & $3$     & $-0.94001(*)$ & $0.41$      & $-0.94001$      & $0.01$          & $0.00$         \\
Soybean$\_$(Small)                                                          & $47$      & $20$    & $-48.5653$   & $600$       & $-44.3561$      & $0.12$          & $0.09$         \\
Daily$\_$Demand$\_$Forecasting$\_$Orders                                    & $60$      & $12$    & $-136.6481$  & $600$       & $-125.9895$     & $0.07$          & $0.08$         \\
Cervical$\_$Cancer$\_$Behavior$\_$Risk                                      & $72$      & $19$    & $-11.748$    & $600$       & $-8.1748$       & $0.12$          & $0.33$         \\
Hepatitis                                                                   & $80$      & $19$    & $-168.6379$  & $600$       & $-165.6924$     & $0.51$          & $0.02$         \\
Fertility                                                                   & $100$     & $9$     & $7.4072(*)$  & $76.72$     & $8.3906$        & $0.02$          & $0.11$         \\
Zoo                                                                         & $101$     & $16$    & $-52.8589$   & $600$       & $-50.5052$      & $0.05$          & $0.04$         \\
Breast$\_$Cancer$\_$Coimbra                                                 & $116$     & $9$     & $-57.447(*)$ & $11.19$     & $-57.1875$      & $0.05$          & $0.00$         \\
Primary$\_$Tumor                                                            & $132$     & $17$    & $-16.0921$   & $600$       & $-12.1534$      & $0.07$          & $0.26$         \\
Higher$\_$Education$\_$Students$\_$Performance$\_$Evaluation                & $145$     & $31$    & $-56.8203$   & $600$       & $-46.5442$      & $0.24$          & $0.20$         \\
Iris                                                                        & $150$     & $4$     & $-9.6767(*)$ & $1.09$      & $-9.3734$       & $0.02$          & $0.03$         \\
Hayes-Roth                                                                  & $160$     & $4$     & $-0.29403(*)$ & $0.53$      & $-0.1878$       & $0.02$          & $0.09$         \\
Wine                                                                        & $178$     & $13$    & $-147.6318$  & $600$       & $-145.8124$     & $0.43$          & $0.01$         \\
Breast$\_$Cancer$\_$Wisconsin$\_$(Prognostic)                               & $194$     & $33$    & $-556.5208$  & $600$       & $-438.696$      & $1.11$          & $0.24$         \\
Parkinsons                                                                  & $195$     & $20$    & $-332.9395$  & $600$       & $-254.6942$     & $0.41$          & $0.27$         \\
Image$\_$Segmentation                                                       & $210$     & $19$    & $-280.9973$  & $600$       & $-190.3849$     & $0.29$          & $0.38$         \\
Glass$\_$Identification                                                     & $214$     & $9$     & $-61.7463$   & $600$       & $-61.0086$      & $0.21$          & $0.01$         \\
Soybean$\_$(Large)                                                          & $266$     & $35$    & $-70.8578$   & $600$       & $-61.1069$      & $0.37$          & $0.15$         \\
SPECTF$\_$Heart                                                             & $267$     & $44$    & $-61.7682$   & $600$       & $-51.4834$      & $1.47$          & $0.18$         \\
SPECT$\_$Heart                                                              & $267$     & $22$    & $28.2914$    & $600$       & $34.7918$       & $0.11$          & $0.20$         \\
Statlog$\_$(Heart)                                                          & $270$     & $13$    & $-105.021$   & $600$       & $-102.7533$     & $0.27$          & $0.02$         \\
Heart$\_$Disease                                                            & $297$     & $13$    & $-104.9455$  & $600$       & $-102.7195$     & $0.22$          & $0.02$         \\
Heart$\_$Failure$\_$Clinical$\_$Records                                     & $299$     & $12$    & $-236.0395$  & $600$       & $-179.6787$     & $0.23$          & $0.27$         \\
Haberman's$\_$Survival                                                      & $306$     & $3$     & $-2.2252(*)$ & $0.02$      & $-2.2252$       & $0.03$          & $0.00$         \\
Vertebral$\_$Column                                                         & $310$     & $6$     & $-54.6542$   & $600$       & $-27.9576$      & $0.12$          & $0.63$         \\
Ecoli                                                                       & $336$     & $7$     & $-4.7142(*)$ & $1.07$      & $-4.221$        & $0.06$          & $0.09$         \\
Land$\_$Mines                                                               & $338$     & $3$     & $0.35537(*)$ & $0.16$      & $0.61624$       & $0.03$          & $0.18$         \\
Liver$\_$Disorders                                                          & $345$     & $5$     & $-9.6506(*)$ & $0.29$      & $-9.58$         & $0.05$          & $0.01$         \\
Ionosphere                                                                  & $351$     & $33$    & $80.2523$    & $600$       & $87.3991$       & $0.87$          & $0.08$         \\
Dermatology                                                                 & $358$     & $34$    & $-205.4$     & $600$       & $-195.2633$     & $0.48$          & $0.05$         \\
Auto$\_$MPG                                                                 & $392$     & $7$     & $-66.0296$   & $600$       & $-65.7233$      & $0.24$          & $0.00$         \\
User$\_$Knowledge$\_$Modeling                                               & $403$     & $5$     & $-0.56247(*)$ & $0.21$      & $-0.20268$      & $0.05$          & $0.26$         \\
Real$\_$Estate$\_$Valuation                                                 & $414$     & $6$     & $-70.8655$   & $600$       & $-58.7011$      & $0.20$          & $0.18$         \\
MONK's$\_$Problems                                                          & $432$     & $6$     & $-4.1125(*)$ & $52.93$     & $-3.462$        & $0.07$          & $0.14$         \\
Wholesale$\_$customers                                                      & $440$     & $7$     & $-30.8299$   & $600$       & $-26.1566$      & $0.25$          & $0.16$         \\
Breast$\_$Cancer$\_$Wisconsin$\_$(Diagnostic)                               & $569$     & $30$    & $-516.6465$  & $600$       & $-420.3534$     & $2.43$          & $0.21$         \\
Balance$\_$Scale                                                            & $625$     & $4$     & $1.1087(*)$  & $0.16$      & $1.2756$        & $0.10$          & $0.08$         \\
Breast$\_$Cancer$\_$Wisconsin$\_$(Original)                                 & $683$     & $9$     & $2.0562(*)$  & $29.06$     & $2.8805$        & $0.24$          & $0.24$         \\
Statlog$\_$(Australian$\_$Credit$\_$Approval)                               & $690$     & $14$    & $-236.1593$  & $600$       & $-207.0687$     & $0.86$          & $0.13$         \\
National$\_$Poll$\_$on$\_$Healthy$\_$Aging$\_$(NPHA)                        & $714$     & $14$    & $-19.859$    & $600$       & $-17.265$       & $0.39$          & $0.13$         \\
Absenteeism$\_$at$\_$work                                                   & $740$     & $19$    & $-124.532$   & $600$       & $-118.4592$     & $1.35$          & $0.05$         \\
Blood$\_$Transfusion$\_$Service$\_$Center                                   & $748$     & $3$     & $-20.546(*)$ & $0.18$      & $-20.5458$      & $0.30$          & $0.00$         \\
Energy$\_$Efficiency                                                        & $768$     & $7$     & $-64.8916$   & $600$       & $-64.2346$      & $0.72$          & $0.01$         \\
Mammographic$\_$Mass                                                        & $830$     & $5$     & $-20.8851(*)$ & $2.41$      & $-20.632$       & $0.37$          & $0.01$         \\
Statlog$\_$(Vehicle$\_$Silhouettes)                                         & $845$     & $18$    & $-126.1669$  & $600$       & $-122.8075$     & $3.13$          & $0.03$         \\
\hline
	\end{tabular}
	\caption{Full UCI dataset results (part 1 of 2).}
	\label{tab:full_uci1}
\end{table*}

\begin{table*}[!t]
	\centering
	 \footnotesize
	\setlength{\tabcolsep}{2.85pt}%
	\renewcommand{\arraystretch}{0.86}%
	\begin{tabular}{|l|c|c|c|c|c|c|c|}
		\hline
		\multirow{2}{*}{Dataset} & \multirow{2}{*}{$n$} & \multirow{2}{*}{$r$} & \multicolumn{2}{c|}{Branch-and-Bound}   & \multicolumn{2}{c|}{Continuous Relaxation}   & \multirow{2}{*}{GAP} \\  \cline{4-7}
		 & & & LB &  Time & UB & Time &  \\ 
		\hline
Maternal$\_$Health$\_$Risk                                                  & $1014$    & $6$     & $-13.7333(*)$ & $510.77$    & $-13.6184$      & $0.46$          & $0.01$         \\
Concrete$\_$Compressive$\_$Strength                                         & $1030$    & $8$     & $-24.6657$   & $600$       & $-22.9041$      & $0.59$          & $0.07$         \\
Diabetic$\_$Retinopathy$\_$Debrecen                                         & $1151$    & $18$    & $-139.2206$  & $600$       & $-135.1317$     & $4.17$          & $0.03$         \\
Website$\_$Phishing                                                         & $1353$    & $9$     & $17.5505$    & $600$       & $18.3885$       & $0.94$          & $0.04$         \\
Banknote$\_$Authentication                                                  & $1372$    & $4$     & $-3.9593(*)$ & $2.48$      & $-3.8468$       & $0.56$          & $0.02$         \\
Hepatitis$\_$C$\_$Virus$\_$(HCV)$\_$for$\_$Egyptian$\_$patients             & $1385$    & $28$    & $-531.9193$  & $600$       & $-358.6365$     & $5.68$          & $0.39$         \\
Contraceptive$\_$Method$\_$Choice                                           & $1473$    & $9$     & $-42.9772$   & $600$       & $-41.4642$      & $1.57$          & $0.04$         \\
Yeast                                                                       & $1484$    & $8$     & $-4.8033(*)$ & $156.85$    & $-4.1933$       & $0.69$          & $0.11$         \\
Airfoil$\_$Self-Noise                                                       & $1503$    & $5$     & $-73.1905$   & $600$       & $-57.2507$      & $1.49$          & $0.24$         \\
Drug$\_$Consumption$\_$(Quantified)                                         & $1885$    & $12$    & $0.18583$    & $600$       & $2.089$         & $2.33$          & $0.89$         \\
Steel$\_$Plates$\_$Faults                                                   & $1941$    & $27$    & $-695.4052$  & $600$       & $-344.2675$     & $7.07$          & $0.67$         \\
Auction$\_$Verification                                                     & $2043$    & $7$     & $-40.9852$   & $600$       & $-40.1057$      & $2.15$          & $0.02$         \\
Cardiotocography                                                            & $2126$    & $20$    & $-173.6264$  & $600$       & $-147.7546$     & $8.46$          & $0.16$         \\
AIDS$\_$Clinical$\_$Trials$\_$Group$\_$Study$\_$175                         & $2139$    & $23$    & $-269.8734$  & $600$       & $-258.93$       & $10.31$         & $0.04$         \\
National$\_$Health$\_$and$\_$Nutrition$\_$Health$\_$Survey$\_$2013-2014$\_$ & $2278$    & $7$     & $-43.3426$   & $600$       & $-42.9178$      & $4.22$          & $0.01$         \\
Statlog$\_$(Image$\_$Segmentation)                                          & $2310$    & $19$    & $-270.7187$  & $600$       & $-184.6101$     & $9.37$          & $0.38$         \\
Iranian$\_$Churn                                                            & $3150$    & $13$    & $-161.5337$  & $600$       & $-153.2374$     & $13.87$         & $0.05$         \\
Ozone$\_$Level$\_$Detection                                                 & $3695$    & $72$    & $-1021.1334$ & $601$       & $-963.1435$     & $86.59$         & $0.06$         \\
Rice$\_$(Cammeo$\_$and$\_$Osmancik)                                         & $3810$    & $7$     & $-90.6916$   & $600$       & $-78.7753$      & $15.77$         & $0.14$         \\
Predict$\_$Students'$\_$Dropout$\_$and$\_$Academic$\_$Success               & $4424$    & $36$    & $-459.3387$  & $600$       & $-437.3949$     & $25.70$         & $0.05$         \\
Spambase                                                                    & $4601$    & $57$    & $-809.8656$  & $600$       & $-785.3801$     & $32.93$         & $0.03$         \\
Waveform$\_$Database$\_$Generator$\_$(Version$\_$1)                         & $5000$    & $21$    & $-14.459$    & $600$       & $-7.214$        & $6.27$          & $0.61$         \\
Page$\_$Blocks$\_$Classification                                            & $5473$    & $10$    & $-105.606$   & $600$       & $-92.8885$      & $19.13$         & $0.13$         \\
Optical$\_$Recognition$\_$of$\_$Handwritten$\_$Digits                       & $5620$    & $62$    & $-5.1735$    & $601$       & $23.6513$       & $45.87$         & $1.87$         \\
Parkinsons$\_$Telemonitoring                                                & $5875$    & $19$    & $-284.7582$  & $600$       & $-234.9981$     & $22.02$         & $0.19$         \\
Statlog$\_$(Landsat$\_$Satellite)                                           & $6435$    & $36$    & $-106.3612$  & $600$       & $-90.3697$      & $39.49$         & $0.16$         \\
Wine$\_$Quality                                                             & $6497$    & $11$    & $-95.0783$   & $600$       & $-94.1656$      & $43.71$         & $0.01$         \\
Musk$\_$(Version$\_$2)                                                      & $6598$    & $166$   & $-332.7221$  & $601$       & $-245.6389$     & $814.92$        & $0.30$         \\
Taiwanese$\_$Bankruptcy$\_$Prediction                                       & $6819$    & $24$    & $-6.52$      & $601$       & $-1.6099$       & $25.65$         & $0.97$         \\
ISOLET                                                                      & $7797$    & $617$   & $1296.2834$  & $609$       & N.A.            & N.A.            & N.A.           \\
Combined$\_$Cycle$\_$Power$\_$Plant                                         & $9568$    & $4$     & $-20.0107$   & $600$       & $-19.0065$      & $29.76$         & $0.05$         \\
Electrical$\_$Grid$\_$Stability$\_$Simulated$\_$Data$\_$                    & $10000$   & $11$    & $-23.6946$   & $600$       & $-21.217$       & $31.47$         & $0.11$         \\
Pen-Based$\_$Recognition$\_$of$\_$Handwritten$\_$Digits                     & $10992$   & $16$    & $3.4228$     & $600$       & $8.5114$        & $38.97$         & $0.73$         \\
Phishing$\_$Websites                                                        & $11055$   & $30$    & $82.465$     & $600$       & $93.1696$       & $138.58$        & $0.12$         \\
Dry$\_$Bean                                                                 & $13611$   & $16$    & $-391.885$   & $601$       & $-240.8034$     & $153.03$        & $0.48$         \\
EEG$\_$Eye$\_$State                                                         & $14980$   & $14$    & $-183.4523$  & $600$       & $-158.72$       & $181.37$        & $0.14$         \\
HTRU2                                                                       & $17898$   & $8$     & $-47.5217$   & $601$       & $-46.257$       & $214.40$        & $0.03$         \\
MAGIC$\_$Gamma$\_$Telescope                                                 & $19020$   & $10$    & $-47.9973$   & $600$       & $-45.7039$      & $363.05$        & $0.05$         \\
Polish$\_$Companies$\_$Bankruptcy                                           & $19967$   & $65$    & $-1361.4574$ & $604$       & $-922.5883$     & $795.90$        & $0.38$         \\
Letter$\_$Recognition                                                       & $20000$   & $16$    & $-8.9559$    & $600$       & $-5.1585$       & $184.55$        & $0.47$         \\
Superconductivty$\_$Data                                                    & $21263$   & $81$    & $-1037.7763$ & $604$       & $-905.8072$     & $3503.56$       & $0.14$         \\
NATICUSdroid$\_$(Android$\_$Permissions)                                    & $29332$   & $85$    & $128.0883$   & $600$       & $165.7042$      & $1547.99$       & $0.25$         \\
Default$\_$of$\_$Credit$\_$Card$\_$Clients                                  & $30000$   & $23$    & $-270.6213$  & $600$       & $-186.0217$     & $1054.13$       & $0.37$         \\
Online$\_$News$\_$Popularity                                                & $39644$   & $55$    & $-1126.6942$ & $600$       & N.A.            & N.A.            & N.A.           \\
\hline
	\end{tabular}
	\caption{Full UCI dataset results (part 2 of 2).}
	\label{tab:full_uci2}
\end{table*}

\subsection{Odd Cycle Packing (OCP)}
We also apply our algorithm to the odd cycle packing problem as this
problem was used to show that MAXDET is NP-hard; see \cite{di2015largest}.
Given a simple undirected graph, we want to find a maximum 
family of vertex-disjoint odd cycles. 
Given a graph $G$ with node-edge incidence matrix $A_G$, for every
odd cycle $C$ of $G$, we get that the square submatrix of $A_G$ with 
rows corresponding to the nodes of $C$
and columns corresponding to the edges of $G$ has determinant $\pm 2$.
See e.g.,~\cite{di2015largest}.

We choose a random simple undirected graph $G$ with incidence matrix
$V^T=A_G$. We then apply our relaxed problem of maximizing the
determinant. We use the following MATLAB code to generate the random graphs. The results are shown in \Cref{tab:full_ocp}. We report the objective values in log base $2$.
\begin{lstlisting}[style=Matlab-editor]
	G = graph(true(r), 'omitselfloops');
	p = randperm(numedges(G), n);
	G = graph(G.Edges(p, :));
	V = full(abs(incidence(G)))';
\end{lstlisting}

\begin{table*}[!t]
	\centering
	\footnotesize
	\setlength{\tabcolsep}{2.85pt}%
	\renewcommand{\arraystretch}{0.86}%
	\begin{tabular}{|c|c|c|c|c|c|c|}
		\hline
		 \multirow{2}{*}{$n$} & \multirow{2}{*}{$r$} & \multicolumn{2}{c|}{Branch-and-Bound}   & \multicolumn{2}{c|}{Continuous Relaxation}   & \multirow{2}{*}{GAP} \\  \cline{3-6}
		& & LB &  Time & UB & Time &  \\ 
		\hline
$10$      & $5$     & $2(*)$       & $0.01$     & $4.3399$        & $0.00$          & $0.51$          \\
$11$      & $6$     & $4(*)$       & $0.00$     & $4.7549$        & $0.00$          & $0.13$          \\
$12$      & $6$     & $4(*)$       & $0.00$     & $4.937$         & $0.00$          & $0.16$          \\
$13$      & $7$     & $4(*)$       & $0.00$     & $5.4449$        & $0.00$          & $0.23$          \\
$14$      & $7$     & $4(*)$       & $0.00$     & $5.5565$        & $0.00$          & $0.25$          \\
$15$      & $8$     & $4(*)$       & $0.01$     & $5.7443$        & $0.00$          & $0.28$          \\
$16$      & $8$     & $4(*)$       & $0.01$     & $6.2$           & $0.00$          & $0.34$          \\
$17$      & $9$     & $4(*)$       & $0.03$     & $6.5274$        & $0.01$          & $0.38$          \\
$18$      & $9$     & $6(*)$       & $0.03$     & $7.3216$        & $0.00$          & $0.16$          \\
$19$      & $10$    & $2(*)$       & $0.40$     & $6.2492$        & $0.01$          & $0.76$          \\
$20$      & $10$    & $4(*)$       & $0.13$     & $7.3204$        & $0.01$          & $0.47$          \\
$21$      & $11$    & $4(*)$       & $0.23$     & $7.5539$        & $0.01$          & $0.49$          \\
$22$      & $11$    & $6(*)$       & $0.07$     & $8.4524$        & $0.01$          & $0.28$          \\
$23$      & $12$    & $4(*)$       & $0.51$     & $7.6259$        & $0.01$          & $0.50$          \\
$24$      & $12$    & $6(*)$       & $0.29$     & $9.008$         & $0.01$          & $0.34$          \\
$25$      & $13$    & $6(*)$       & $0.55$     & $8.4247$        & $0.02$          & $0.28$          \\
$26$      & $13$    & $4(*)$       & $1.99$     & $8.2512$        & $0.02$          & $0.56$          \\
$27$      & $14$    & $6(*)$       & $1.29$     & $9.4815$        & $0.01$          & $0.38$          \\
$28$      & $14$    & $6(*)$       & $2.14$     & $9.7649$        & $0.01$          & $0.40$          \\
$29$      & $15$    & $6(*)$       & $7.71$     & $9.7195$        & $0.01$          & $0.40$          \\
$30$      & $15$    & $8(*)$       & $8.28$     & $11.4575$       & $0.01$          & $0.31$          \\
$31$      & $16$    & $8(*)$       & $3.21$     & $11.5478$       & $0.02$          & $0.32$          \\
$32$      & $16$    & $6(*)$       & $26.26$    & $11.4689$       & $0.01$          & $0.54$          \\
$33$      & $17$    & $8(*)$       & $27.10$    & $12.0676$       & $0.02$          & $0.35$          \\
$34$      & $16$    & $10(*)$      & $2.49$     & $12.6579$       & $0.01$          & $0.21$          \\
$35$      & $18$    & $6(*)$       & $206.69$   & $11.5627$       & $0.07$          & $0.54$          \\
$36$      & $18$    & $8(*)$       & $69.62$    & $13.2113$       & $0.02$          & $0.43$          \\
$37$      & $19$    & $6(*)$       & $150.38$   & $11.6104$       & $0.05$          & $0.55$          \\
$38$      & $19$    & $6(*)$       & $336.25$   & $12.1435$       & $0.07$          & $0.58$          \\
$39$      & $20$    & $8(*)$       & $440.52$   & $13.1572$       & $0.05$          & $0.43$          \\
$40$      & $20$    & $8$          & $600$      & $14.7234$       & $0.02$          & $0.53$          \\
$41$      & $21$    & $10(*)$      & $243.98$   & $14.8876$       & $0.02$          & $0.35$          \\
$42$      & $21$    & $8$          & $600$      & $13.6182$       & $0.06$          & $0.46$          \\
$43$      & $22$    & $8$          & $600$      & $15.1925$       & $0.02$          & $0.55$          \\
$44$      & $22$    & $6$          & $600$      & $13.4259$       & $0.12$          & $0.67$          \\
\hline
	\end{tabular}
	\caption{Full OCP results in $\log_2$ base.}
	\label{tab:full_ocp}
\end{table*}

\section{Statements and Declarations}
\textbf{Competing Interests:} This work was supported in part by the Air Force Office of Scientific Research under award number FA9550-23-1-0508 (Hao Hu), by the National Science Foundation under grant DMS 2348720 (Hugo J. Woerdeman), and by the Natural Sciences and Engineering Research Council of Canada (Henry Wolkowicz). The authors have no relevant financial or non-financial interests to disclose.

\textbf{Data availability:} The datasets analyzed during the current
study were generated randomly, and the procedure for their generation is
fully described in the article. We also use previously published
datasets that are publicly available and are fully cited within the article.


\bibliographystyle{plain}
\bibliography{maxdet_refs}

\end{document}